\documentclass[12pt]{amsart}
\usepackage{amsthm, amssymb, color}
\usepackage{amsfonts, amscd}
\usepackage{epsfig,multicol}

\usepackage{tikz}
\usepackage{tikz-cd}
\usepackage{tikz-3dplot}
\usetikzlibrary{shapes.geometric, calc}

\usetikzlibrary{math}

\theoremstyle{plain} \numberwithin{equation}{section}
\newtheorem{theorem}{Theorem}[section]
\newtheorem{corollary}[theorem]{Corollary}
\newtheorem{proposition}[theorem]{Proposition}
\newtheorem{lemma}[theorem]{Lemma}

\theoremstyle{definition}

\newtheorem{example}[theorem]{Example}
\newtheorem*{remark}{Remark}

\def\Z{\mathbb Z}
\def\C{\mathbb C}
\def\R{\mathbb R}
\def\Q{\mathbb Q}

\def\x{\mathbf{x}}
\def\n{\mathbf{n}}
\def\z{\mathbf{z}}

\DeclareMathOperator{\Sign}{Sign}

\begin{document}

\title[Symmetric matrices defined by plane vector sequences]{Symmetric matrices defined by \\
plane vector sequences}

\author[M. Masuda]{Mikiya Masuda}
\address{Osaka Central Advanced Mathematical Institute, Osaka Metropolitan University, Sumiyoshi-ku, Osaka 558-8585, Japan.}
\email{mikiyamsd@gmail.com}

\date{\today}

\begin{abstract}
Motivated by a work of Fu-So-Song, we associate a symmetric matrix $A$ to a plane vector sequence $v$ and give a formula to find the signature of $A$ in terms of the sequence $v$. When $A$ is nonsingular, we interpret the relation between $A$ and $A^{-1}$ from a topological viewpoint. Finally, we associate an omnioriented quasitoric orbifold $X$ of real dimension four to the sequence $v$ and show that $A^{-1}$ is the intersection matrix of the characteristic suborbifolds of $X$. 
\end{abstract}

\maketitle 

\section{Introduction}

Let $v=(v_0,v_1,\dots,v_n,v_{n+1})$ $(n\ge 1)$ be a sequence of plane vectors where
\begin{equation} \label{eq:vi}
\text{$v_0=\begin{pmatrix}0\\
1\end{pmatrix}$, $v_i=\begin{pmatrix}a_i\\
b_i\end{pmatrix}\in \R^2$ for $1\le i\le n$, and $v_{n+1}=\begin{pmatrix}1\\
0\end{pmatrix}$}.
\end{equation}
Motivated by Fu-So-Song \cite{fu-so-so23}, we form a symmetric matrix $A$ of order $n$ with $a_ib_j$ as $(i,j)$ entry for $i\le j$ from the sequence $v$, that is
\begin{equation} \label{eq:matrix_A}
A=\begin{pmatrix}
a_1b_1& a_1b_2& \dots &a_1b_{n-1}& a_1b_n\\
a_1b_2&a_2b_2&\dots&a_2b_{n-1}&a_2b_n\\
\vdots&\vdots&\ddots& \vdots&\vdots\\
a_1b_{n-1}&a_2b_{n-1}&\dots&a_{n-1}b_{n-1}&a_{n-1}b_{n}\\
a_1b_n&a_2b_n&\dots&a_{n-1}b_n&a_nb_n
\end{pmatrix}
\end{equation}
One can easily find that 
\begin{equation} \label{eq:determinant_A}
\det A=(-1)^{n+1}\prod_{i=0}^{n}\det(v_i,v_{i+1})=a_1b_n\prod_{i=1}^{n-1}(a_{i+1}b_i-a_ib_{i+1}).
\end{equation}
Since $A$ is symmetric, its eigenvalues are all real numbers and the signature of $A$, denoted by $\Sign(A)$, is defined to be the number of positive eigenvalues of $A$ minus the number of negative eigenvalues of $A$. 
If $\det(v_i,v_{i+1})=0$ for some $i$, then one can see that the symmetric matrix $A'$ defined by a vector sequence $v'$ with $v_i$ or $v_{i+1}$ removed from $v$ has the same signature as $A$. Therefore, we assume 
\begin{equation} \label{eq:assumption}
\text{$\det(v_i,v_{i+1})\not=0$ for every $i=0,\dots,n$, i.e. $\det A\not=0$}. 
\end{equation}

We understand that $v_i$ rotates toward $v_{i+1}$ with angle less than $\pi$ for each $i$. The vectors $v_0,v_1,\dots,v_n, v_{n+1},v_0$ rotate around the origin of $\R^2$ possibly many times and may go back and forth. We introduce two invariants of the sequence $v=(v_0,v_1,\dots,v_n,v_{n+1})$:
\begin{equation}
\begin{split}
R(v)&=\text{ the rotation number of $v$ around the origin of $\R^2$},\\
S(v)&=1+\sum_{i=0}^n{\rm sgn}(\det(v_i,v_{i+1}))
\end{split}
\end{equation}
where ${\rm sgn}(a)$ denotes $1$ if $a>0$ and $-1$ if $a<0$, and $1$ at the right hand side above comes from ${\rm sgn}(\det(v_{n+1},v_0))$. Our first main result is the following. 

\begin{theorem} \label{theo:main}
$\Sign(A)=4R(v)-S(v)$. 
\end{theorem}

\begin{remark}
For an almost complex compact manifold $X$ of complex dimension two, it is well-known that $T(X)=\frac{1}{4}(\Sign(X)+\chi(X))$ where $T(X)$ is the Todd genus of $X$, $\Sign(X)$ is the signature of $X$, and $\chi(X)$ is the Euler characteristic of $X$. Hence we have $\Sign(X)=4T(X)-\chi(X)$. Theorem~\ref{theo:main} is a combinatorial analogue of this identity. 
\end{remark}

\begin{example} \label{exam:1}
(1) We take $n=2$ and $v_1=\begin{pmatrix}-1\\
-1\end{pmatrix}$, $v_2=\begin{pmatrix}0\\
-1\end{pmatrix}$. Then $A=\begin{pmatrix} 1& 1\\
1&0\end{pmatrix}$ (so $\Sign(A)=0$) while $R(v)=1$ and $S(v)=4$. 

(2) We take $n=2$ and $v_1=\begin{pmatrix}-1\\
-1\end{pmatrix}$, $v_2=\begin{pmatrix}-2\\
-1\end{pmatrix}$. Then $A=\begin{pmatrix} 1& 1\\
1&2\end{pmatrix}$ (so $\Sign(A)=2$) while $R(v)=1$ and $S(v)=2$. 

(3) We take $n=4$ and 
\[
v_1=\begin{pmatrix}-1\\
-1\end{pmatrix},\ v_2=\begin{pmatrix}1\\
0\end{pmatrix},\ v_3=\begin{pmatrix}0\\
1\end{pmatrix},\ v_4=\begin{pmatrix}-1\\
-1\end{pmatrix}
\]
Then 
\[
A=\begin{pmatrix} 1&0&-1&1\\
0&0&1&-1\\
-1&1&0&0\\
1&-1&0&1\end{pmatrix}
\]
One can find $\Sign(A)=2$ while $R(v)=2$ and $S(v)=6$. 
\end{example}

It follows from the definition of the signature, $\Sign(A)$ must be $n-2k$ for some $0\le k\le n$. The following corollary shows that any such value can be realized as $\Sign(A)$. 

\begin{corollary} \label{coro:positive_definite}
Suppose that all $v_i$'s lie on the upper half plane of $\R^2$. If the sequence $v=(v_0, v_1,\dots,v_{n+1})$ rotates counterclockwise up to $v_k$ for some $0\le k\le n$ and clockwise from $v_k$ around the origin, then $\Sign(A)=n-2k$. 
In particular, $A$ is positive definite {\rm (}i.e. $\Sign(A)=n${\rm )} if $v=(v_0, v_1,\dots,v_n, v_{n+1})$ rotates around the origin clockwise. 
\end{corollary} 

\begin{proof}
In this case $R(v)=0$ and $S(v)=-n+2k$ since $\det(v_i,v_{i+1})<0$ if and only if $k\le i\le n$. 
\end{proof}

A simple calculation shows that the inverse $A^{-1}$ of our matrix $A$ is a nonsingular tridiagonal symmetric matrix with entries in the superdiagonal (i.e. the diagonal just above the main diagonal) of $A^{-1}$ being all nonzero (Proposition~\ref{prop:inverse_A}). Conversely, any such nonsingular tridiagonal symmetric matrix can be obtained as $A^{-1}$ of our matrix $A$ (Proposition~\ref{prop:bijection}). 
We observe the relation between $A$ and $A^{-1}$ from a topological viewpoint in Section~\ref{sect:4}. This observation is made in \cite{fu-so-so24} from a different perspective when each $v_i$ is a primitive integer vector and the sequence $v$ rotates around the origin once counterclockwise. 

To any integer vector sequence $v$ satisfying the condition \eqref{eq:assumption}, one can associate an omnioriented quasitoric orbifold $X$ of real dimension four. When $\{v_i, v_{i+1}\}$ forms a basis of $\Z^2$ for any $0\le i\le n$, $X$ is smooth. For instance, the $X$ associated to the integer vector sequence $v$ in Example~\ref{exam:1} is respectively $\C P^2\#\overline{\C P^2}$, $\C P^2\#{\C P^2}$, $3\C P^2\#\overline{\C P^2}$, where $\overline{\C P^2}$ denotes $\C P^2$ with the opposite orientation. Note that $\Sign(X)=\Sign(A)$. 
In general, the characteristic suborbifolds $X_1,\dots,X_n$ of $X$ corresponding to $v_1,\dots,v_n$ define a basis $\alpha_1,\dots,\alpha_n$ of $H_2(X;\R)$. We prove that the intersection matrix of $X_1,\dots,X_n$ is the inverse of $A$ (Theorem~\ref{theo:main2}). This result together with the observation made in Sections~\ref{sect:3} and~\ref{sect:4} implies the following. 

\begin{theorem} \label{theo:main2_intro}
Let $X$ be an omnioriented quasitoric orbifold associated to an integer vector sequence $v$ satisfying \eqref{eq:assumption} and let $\alpha_1^*,\dots,\alpha_n^*$ be the Kronecker dual to $\alpha_1,\dots,\alpha_n$, i.e. $\alpha_i^*\in H^2(X;\R)$ and $\langle \alpha_i^*,\alpha_j\rangle$ is the Kronecker delta $\delta_{ij}$. Then $a_ib_j=\langle \alpha_i^*\cup\alpha_j^*,[X]\rangle$ for $i\le j$, i.e. $A=(\langle \alpha_i^*\cup\alpha_j^*,[X]\rangle)$, where $[X]$ denotes the fundamental class of $X$.
\end{theorem} 

This result is proved in \cite{fu-so-so23} (by a different method) also when each $v_i$ is a primitive integer vector and the sequence $v$ rotates around the origin once counterclockwise (in this case $X$ is a genuine toric orbifold). 

\medskip
\noindent
{\bf Acknowledgment.} The author thanks Xin Fu, Tseleung So, and Jongbaek Song for helpful discussions and Hiraku Abe for informing him of the reference \cite{lu-ti92}. He is supported in part by the HSE University Basic Research Program, JSPS Grant-in-Aid for Scientific Research 22K03292, and Osaka Central Advanced Mathematical Institute (MEXT Joint Usage/Research Center on Mathematics and Theoretical Physics JPMXP0723833165).
\medskip

\section{Proof of Theorem~\ref{theo:main}}

For a while, we allow $v_i$'s $(1\le i\le n)$ to be arbitrary plane (even zero) vectors. 
The following is \eqref{eq:determinant_A} in the introduction. 

\begin{lemma} \label{lemm:determinant_A}
$\det A=a_1b_n\prod_{i=1}^{n-1}(a_{i+1}b_i-a_ib_{i+1})$. 
\end{lemma}

\begin{proof}
We regard $a_i$'s and $b_i$'s as indeterminates and prove the lemma by induction on $n$. When $n=1$, the lemma is obvious. Suppose that $n\ge 2$ and the lemma holds for $n-1$. Adding the $(n-1)$st column multiplied by $-b_n/b_{n-1}$ to the $n$-th column in \eqref{eq:matrix_A}, we obtain 
\begin{equation} \label{eq:determinant_A2}
\det A=\begin{vmatrix}
a_1b_1& a_1b_2& \dots &a_1b_{n-1}& 0\\
a_1b_2&a_2b_2&\dots&a_2b_{n-1}&0\\
\vdots&\vdots&\ddots& \vdots&\vdots\\
a_1b_{n-1}&a_2b_{n-1}&\dots&a_{n-1}b_{n-1}&0\\
a_1b_n&a_2b_n&\dots&a_{n-1}b_n&a_nb_n-a_{n-1}b_n\cdot \frac{b_n}{b_{n-1}}
\end{vmatrix}
\end{equation}
Since the $(n,n)$ entry above is $(a_nb_{n-1}-a_{n-1}b_n)b_n/b_{n-1}$, it follows from the induction assumption that \eqref{eq:determinant_A2} turns into 
\[
\begin{split}
\det A&=a_1b_{n-1}\left(\prod_{i=1}^{n-2}(a_{i+1}b_i-a_ib_{i+1})\right)(a_nb_{n-1}-a_{n-1}b_n)\frac{b_n}{b_{n-1}}\\
&=a_1b_n\prod_{i=1}^{n-1}(a_{i+1}b_i-a_ib_{i+1}),
\end{split}
\]
proving the lemma. 
\end{proof}

If $v_i$ is zero for some $1\le i\le n$, then entries in the $i$-th row and $i$-th column of $A$ are all zero. Therefore, if $A'$ is $A$ with $i$-th row and $i$-th column removed, then the set of eigenvalues of $A'$ agrees with the set of eigenvalues of $A$ with one $0$ removed. In particular, $\Sign(A')=\Sign(A)$. The $A'$ is associated to the sequence $v$ with $v_i$ removed, so we may remove $v_i$ from $v$ in order to find $\Sign(A)$ whenever $v_i$ is zero. Thus, we may assume that all $v_i$'s are nonzero. 

Suppose that $v_i$ and $v_{i+1}$ are linearly dependent for some $0\le i\le n$. If $i=0$, then $a_1=0$ since $v_0=\begin{pmatrix}0\\
1\end{pmatrix}$; so entries in the $1$st row and $1$st column of $A$ are all zero. Therefore, we may remove $v_1$ from the sequence $v$ by the same reason as above. A similar argument works when $i=n$. 

In the sequel, we assume $v_{i+1}=kv_i$ for some $k\in \R$ and $1\le i\le n-1$. Then, if $P$ denotes a nonsingular matrix of order $n$ with $1$'s on the diagonal entries, $-k$ on the $(i+1,i)$ entry, and $0$ on the other entries, then $A'=PAP^T$ is $A$ with all entries in the $(i+1)$st row and $(i+1)$st column replaced by $0$'s. Since $A'=PAP^T$ and $P$ is nonsingular, Sylvester's law of inertia says that $\Sign(A)=\Sign(A')$. Therefore, we may remove $v_{i+1}$ from the sequence $v$ in this case by the same reason as above. 

Thus, in order to find $\Sign(A)$, we may assume that $v_i$ and $v_{i+1}$ are linearly independent for any $0\le i\le n$ and hence $\det A\not=0$ by Lemma~\ref{lemm:determinant_A}.

We recall the following fact which will be used to prove Theorem~\ref{theo:main}. 

\begin{theorem}[Sylvester] \label{theo:Sylvester}
Let $A$ be a symmetric matrix of order $n$ with real entries such that every principal upper left $k\times k$ minor $\Delta_k$ is nonzero. Then the number of negative eigenvalues of $A$ is equal to the number of sign changes in the sequence
\[
\Delta_0=1,\ \Delta_1,\dots,\ \Delta_n=\det A.
\] 
\end{theorem}

\begin{proof}[Proof of Theorem~\ref{theo:main}]
If we replace $v_i$ by $-v_i$ for some $1\le i\le n$, then the $i$-th column and $i$-th row in $A$ are both multiplied by $-1$; so $\Sign(A)$ remain unchanged by this replacement. On the other hand, the right hand side of the identity in the theorem also remain unchanged by this replacement. Indeed, one can observe that if $v_{i-1}, v_i, v_{i+1}$ rotate counterclockwise (resp. clockwise), then $R(v)$ decreases (resp. increases) by $1$ and $S(v)$ decreases (resp. increases) by $4$, and otherwise $R(v)$ and $S(v)$ remain unchanged; so $4R(v)-S(v)$ remain unchanged in any case. 
Therefore, we may assume that every $v_i$ lies on the upper half plane of $\R^2$. Since $R(v)=0$ in this case, what we have to prove is $\Sign(A)=-S(v)$ when all $v_i$'s lie on the upper half plane of $\R^2$. 

We may assume that $b_k>0$ for any $1\le k\le n$ by rotating $v_k$ slightly when $b_k=0$. The matrix $A$ slightly changes by this slight rotation and the eigenvalues of $A$ slightly change. However, since $\det A\not=0$, the signs of the eigenvalues do not change and hence $\Sign(A)$ remain unchanged. 

It follows from Lemma~\ref{lemm:determinant_A} that 
\[
\Delta_k=a_1b_k\prod_{i=1}^{k-1}(a_{i+1}b_i-a_ib_{i+1})\not=0\quad \text{for $1\le k\le n$},
\] 
where we understand $\prod_{i=1}^{k-1}(a_{i+1}b_i-a_ib_{i+1})=1$ when $k=1$. Then, for $0\le k\le n-1$ we have 
\[
\frac{\Delta_{k+1}}{\Delta_k}=\frac{b_{k+1}}{b_k}(a_{k+1}b_k-a_kb_{k+1})=
-\frac{b_{k+1}}{b_k}\det(v_k,v_{k+1}).
\]
Since $b_{k+1}/b_k>0$, the number of sign changes in the sequence 
\[
\Delta_0=1,\ \Delta_1,\dots,\ \Delta_n=\det A
\] 
is the number of $k$'s for $0\le k\le n-1$ with $\det(v_k,v_{k+1})>0$. Therefore, if we denote the number of negative (resp. positive) eigenvalues of $A$ by $n_-$ (resp. $n_+$), then 
\[
n_-=\sum_{\substack{0\le k\le n-1\\ \det(v_k,v_{k+1})>0}}{\rm sgn}(\det(v_k,v_{k+1}))
\]
by Theorem~\ref{theo:Sylvester}, and 
\[
n_+=-\sum_{\substack{0\le k\le n-1\\ \det(v_k,v_{k+1})<0}}{\rm sgn}(\det(v_k,v_{k+1})). 
\]
Hence 
\begin{equation} \label{eq:n+-}
\Sign(A)=n_+-n_-=-\sum_{k=0}^{n-1}{\rm sgn}(\det(v_k,v_{k+1})).
\end{equation}
Since $b_n>0$ and $v_{n+1}=\begin{pmatrix}1\\
0\end{pmatrix}$, we have $\det(v_n,v_{n+1})<0$ and hence ${\rm sgn}(\det(v_n,v_{n+1}))=-1$. This together with \eqref{eq:n+-} shows $\Sign(A)=-S(v)$, proving the theorem. 
\end{proof}

\begin{remark}
Theorem~\ref{theo:main} also follows from \cite[Corollary 3.3 with $y=1$]{ha-ma03} applied to a multi-fan associated to the sequence $v$. 
\end{remark}

\section{The inverse of the matrix $A$} \label{sect:3}

We assume that our matrix $A$ in \eqref{eq:matrix_A} is nonsingular. 

\begin{proposition} \label{prop:inverse_A}
We set 
\[
\begin{split}
m_i&=\det(v_i,v_{i+1})\quad \text{for\ $0\le i\le n$},\\
s_i&=-\det(v_{i-1},v_{i+1})/m_{i-1}m_i \quad\text{for\ $1\le i\le n$}.
\end{split}
\]
Then the inverse of $A$ is a tridiagonal symmetric matrix given by 
\begin{equation} \label{eq:inverse_A}
{\small\begin{pmatrix} s_1& m_1^{-1}& 0&0&\dots &0\\
m_1^{-1}& s_2& m_2^{-1}&0& \dots&0\\
0&m_2^{-1}&s_3& m_3^{-1}&\dots&0\\
\vdots&\ddots&\ddots&\ddots&\ddots&\vdots\\
0&\dots&0&m_{n-2}^{-1}&s_{n-1}&m_{n-1}^{-1}\\
0&\dots&0&0&m_{n-1}^{-1}&s_n
\end{pmatrix}}
\end{equation}
\end{proposition}

\begin{proof}
By definition 
\begin{equation} \label{eq:cd}
\begin{split}
m_i&=a_ib_{i+1}-a_{i+1}b_i,\\
s_i&=-(a_{i-1}b_{i+1}-a_{i+1}b_{i-1})/(a_{i-1}b_i-a_ib_{i-1})(a_ib_{i+1}-a_{i+1}b_i).
\end{split}
\end{equation}
Let $B$ be the matrix in \eqref{eq:inverse_A}. We shall prove that the product $AB$ is the identity matrix by direct computation. Since $AB$ is a symmetric matrix, it suffices to check the $(i,j)$ entry $(AB)_{i,j}$ of $AB$ for $i\le j$. 

Case 1. The case where $i=j$. In this case
\begin{equation} \label{eq:i=j}
(AB)_{i,i}=a_{i-1}b_im_{i-1}^{-1}+a_ib_is_i+a_ib_{i+1}m_i^{-1}
\end{equation}
where we note that $a_0=0$ when $i=1$ and $b_{n+1}=0$ when $i=n$. 
Using \eqref{eq:cd}, the right hand side of \eqref{eq:i=j} is equal to
\[
\frac{a_{i-1}b_i}{a_{i-1}b_i-a_ib_{i-1}}-\frac{a_ib_i(a_{i-1}b_{i+1}-a_{i+1}b_{i-1})}{(a_{i-1}b_i-a_ib_{i-1})(a_ib_{i+1}-a_{i+1}b_i)}+\frac{a_{i}b_{i+1}}{a_{i}b_{i+1}-a_{i+1}b_{i}}
\]
and a simple calculation shows that this is equal to $1$. 

Case 2. The case where $i<j$. In this case
\begin{equation} \label{eq:i<j}
(AB)_{i,j}=a_ib_{j-1}m_{j-1}^{-1}+a_ib_js_j+a_ib_{j+1}m_j^{-1}
\end{equation} 
where we note that $b_{n+1}=0$ when $j=n$. 
Using \eqref{eq:cd}, the right hand side of \eqref{eq:i<j} is equal to
\[
\frac{a_{i}b_{j-1}}{a_{j-1}b_j-a_jb_{j-1}}-\frac{a_ib_j(a_{j-1}b_{j+1}-a_{j+1}b_{j-1})}{(a_{j-1}b_j-a_jb_{j-1})(a_jb_{j+1}-a_{j+1}b_j)}+\frac{a_{i}b_{j+1}}{a_{j}b_{j+1}-a_{j+1}b_{j}}
\]
and a simple calculation shows that this vanishes. 
\end{proof}

\begin{example} \label{exam:Cartan_matrix}
We take $n=4$ and 
\[
v_1=\begin{pmatrix} 1\\
4\end{pmatrix}, \ v_2=\begin{pmatrix} 2\\
3\end{pmatrix},\ v_3=\begin{pmatrix} 3\\
2\end{pmatrix},\ v_4=\begin{pmatrix} 4\\
1\end{pmatrix}
\]
Then the matrix $A$ in \eqref{eq:matrix_A} associated to this sequence of vectors and its inverse $A^{-1}$ are as follows:
\[
A=\begin{pmatrix} 
4&3&2&1\\
3&6&4&2\\
2&4&6&3\\
1&2&3&4\end{pmatrix}\qquad A^{-1}=\frac{1}{5}\begin{pmatrix}2&-1&0&0\\
-1&2&-1&0\\
0&-1&2&-1\\
0&0&-1&2\end{pmatrix} 
\]

In general, if we take $v_i=\frac{1}{\sqrt{n+1}}\begin{pmatrix}i\\
n+1-i\end{pmatrix}$ for $i=1,\dots,n$, then the sequence $v=(v_0,v_1,\dots,v_n,v_{n+1})$ lies on the upper half plane and rotates clockwise around the origin. Therefore, the matrix $A$ in \eqref{eq:matrix_A} associated to this sequence $v$ is positive definite by Corollary~\ref{coro:positive_definite} and $A^{-1}$ turns out to be the Cartan matrix of type $A_n$. (To the contrary, a formula to find the inverse of a Cartan matrix is given in \cite{lu-ti92} for any Lie type.)
\end{example}

We consider the configuration space 
\begin{equation} \label{eq:V}
V=\left\{\left(v_1=\begin{pmatrix}a_1\\
b_1\end{pmatrix},\dots,v_n=\begin{pmatrix}a_n\\
b_n\end{pmatrix}\right)\in (\R^2)^n\ \middle|\ a_1b_n\prod_{i=1}^{n-1}(a_ib_{i+1}-a_{i+1}b_i)\not=0\right\}
\end{equation}
and the space $W$ of nonsingular tridiagonal symmetric matrices of order $n$ with nonzero entries in the superdiagonal, i.e.
\begin{equation} \label{eq:matrix_B}
W=\left\{B=\begin{pmatrix} p_1& q_1& 0&0&\dots &0\\
q_1& p_2& q_2&0& \dots&0\\
0&q_2&p_3& q_3&\dots&0\\
\vdots&\ddots&\ddots&\ddots&\ddots&\vdots\\
0&\dots&0&q_{n-2}&p_{n-1}&q_{n-1}\\
0&\dots&0&0&q_{n-1}&p_n
\end{pmatrix}\ \middle| \ \begin{split} &\det B\not=0,\\
&q_i\not=0\ \text{$(1\le \forall i\le n-1)$}
\end{split}
\right\}
\end{equation}
The map sending $(v_1,\dots,v_n)\in V$ to $A^{-1}$ defines a map 
$
\Psi\colon V\to W
$
by Proposition~\ref{prop:inverse_A}. The space $V$ admits an action of $\R^*$ sending $\begin{pmatrix}a_i\\
b_i\end{pmatrix}$ to $\begin{pmatrix}ra_i\\
r^{-1}b_i\end{pmatrix}$ $(r\in \R^*)$ simultaneously for every $i$. The matrix $A$ in \eqref{eq:matrix_A} associated to an element of $V$ remains changed under this action, so the map $\Psi$ induces a map 
\[
\widehat{\Psi}\colon V/\R^*\to W. 
\]
The following proposition gives a characterization of nonsingular symmetric matrices of the form \eqref{eq:matrix_A} or a characterization of nonsingular tridiagonal symmetric matrices with nonzero entries in the superdiagonal. 

\begin{proposition} \label{prop:bijection}
The map $\widehat{\Psi}$ is bijective. 
\end{proposition}


\begin{proof}
(Injectivity). We may assume $a_1=1$ for $v\in V$. Since $A^{-1}$ determines $A$, one can see that the $b_1,\dots,b_n$ will be determined from $A^{-1}$ by looking at the first row of $A$. Similarly, since $b_n\not=0$, one can see that $a_2,\dots,a_n$ will also be determined from $A^{-1}$ by looking at the last column of $A$. This shows the injectivity of $\widehat{\Psi}$. 

(Surjectivity). Let $B$ be a matrix in $W$ and let $c_{ij}$ be the cofactor of the $(i,j)$ entry of $B$ (note that $c_{ij}=c_{ji}$ since $B$ is symmetric). Then $B^{-1}=\frac{1}{\det B}(c_{ij})$. What we have to show is the existence of real numbers $a_i,b_j$ $(1\le i,j\le n)$ such that $a_ib_j=\frac{1}{\det B}c_{ij}$ for $i\le j$. We take 
\[
a_i=c_{in}/c_{1n},\quad b_j=c_{1j}/\det B.
\]
Since $a_ib_j=c_{in}c_{1j}/c_{1n}\det B$, it suffices to show that 
\begin{equation} \label{eq:cofactor_equation}
c_{in}c_{1j}=c_{1n}c_{ij}\quad \text{for $i\le j$}. 
\end{equation}
We denote the upper left (resp. lower right) $k\times k$ minor of $B$ by $\Delta_k$ (resp. $\nabla_k$). Then one can easily find that 
\[
c_{ij}=(-1)^{i+j}\Delta_{i-1}\nabla^{n-j}\prod_{\ell=i}^{j-1}q_\ell,
\] 
where we understand $\Delta_0=\nabla^{0}=1$ and $\prod_{\ell=i}^{j-1}q_\ell=1$ when $i=j$. Therefore
\[
\begin{split}
c_{in}c_{1j}&=(-1)^{i+n}\Delta_{i-1}\prod_{\ell=i}^{n-1}q_\ell\cdot (-1)^{1+j}\nabla^{n-j}\prod_{\ell=1}^{j-1}q_\ell\\
c_{1n}c_{ij}&=(-1)^{1+n}\prod_{\ell=1}^{n-1}q_\ell\cdot (-1)^{i+j}\Delta_{i-1}\nabla^{n-j}\prod_{\ell=i}^{j-1}q_\ell,\\
\end{split}
\]
so the identity in \eqref{eq:cofactor_equation} holds. 
\end{proof}

In fact, both $V/\R^*$ and $W$ are topological spaces with a natural topology and $\widehat{\Psi}$ is a homeomorphism. The signs ${\rm sgn}(q_i)$ and $\Sign(B)$ are locally constant continuous functions on $W$. Through the homeomorphism $\widehat{\Psi}$, ${\rm sgn}(q_i)={\rm sgn}(\det(v_i,v_{i+1}))$ and $\Sign(B)=\Sign(A)$. Therefore, the map defined by 
\[
\Phi((v_1,\dots,v_n))= \Big({\rm sgn}(\det(v_1,v_2)),\dots,{\rm sgn}(\det(v_{n-1},v_n)),\Sign(A)\Big)
\]
for $(v_1,\dots,v_n)\in V$ and $A$ in \eqref{eq:matrix_A} associated to $(v_1,\dots,v_n)$ induces a map 
\[
\widehat{\Phi}\colon 
\{\text{connected components of $V/\R^*$}\} \to \{\pm 1\}^{n-1}\times \{n-2k\mid 0\le k\le n\}
\]
where we note that $\Sign(A)$ takes values $n-2k$ $(0\le k\le n)$ (see Corollary~\ref{coro:positive_definite}). 

\begin{proposition}
The map $\widehat{\Phi}$ is bijective. Therefore, $V/\R^*$ and $W$ have $2^{n-1}(n+1)$ many connected components. 
\end{proposition}

\begin{proof}
Sending $(v_1,\dots,v_n)$ to $(\pm v_1,\dots,\pm v_n)$ defines an action of $\{\pm 1\}^{n}/\{\pm 1\}$ on $V/\R^*$, which changes ${\rm sgn}(\det(v_i,v_{i+1}))$'s but does not change $\Sign(A)$. Therefore, it suffices to show that the connected components of $V/\R^*$ with ${\rm sgn}(\det(v_i,v_{i+1}))=1$ for any $1\le i\le n$ are distinguished by $\Sign(A)$. We may further assume that ${\rm sgn}(\det(v_0,v_1))=1$ because if ${\rm sgn}(\det(v_0,v_1))=-1$, then we may consider $(-v_1,\dots,-v_n)$ instead of $(v_1,\dots,v_n)$ since they define the same element in $V/\R^*$. 

The assumption that ${\rm sgn}(\det(v_i,v_{i+1}))=1$ for any $0\le i\le n$ means that the vectors $v_0,v_1,\dots,v_n$ rotate counterclockwise around the origin. Moreover, we have $\det(v_n,v_{n+1})\not=0$. Therefore, the following types occur for the sequence $v=(v_0,v_1,\dots,v_{n+1})$ according to the sign of $\det(v_n,v_{n+1})$ and the rotation number $R(v)$ (see Figures~\ref{fig:n=3} and~\ref{fig:n=4} below):

\smallskip
Type ${\rm I}_k$. $\det(v_n,v_{n+1})<0$ and $R(v)=k$ where $0\le k\le [n/2]$, 

Type ${\rm II}_k$. $\det(v_n,v_{n+1})>0$ and $R(v)=k$ where $1\le k\le [(n+1)/2]$. 

\smallskip

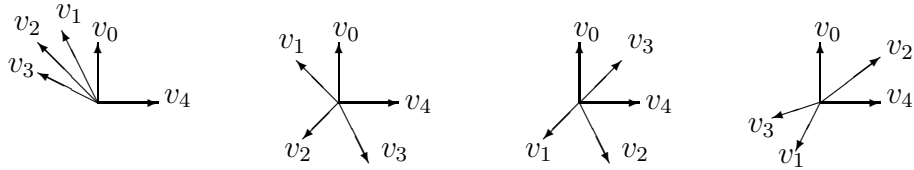
\begin{figure}[hbtp]
\setlength{\unitlength}{0.8cm}
\begin{picture}(14,2.8)
\put(1,1){\vector(1,0){1}}
\put(1,1){\vector(0,1){1}}
\put(1,1){\vector(-1,2){0.6}}
\put(1,1){\vector(-1,1){1}}
\put(1,1){\vector(-2,1){1}}
\put(0.9,2.1){$v_0$}
\put(0.3,2.4){$v_1$}
\put(-0.4,2.2){$v_2$}
\put(-0.5,1.5){$v_3$}
\put(2.1,1){$v_4$}

\put(5,1){\vector(1,0){1}}
\put(5,1){\vector(0,1){1}}
\put(5,1){\vector(-1,1){0.7}}
\put(5,1){\vector(-1,-1){0.6}}
\put(5,1){\vector(1,-2){0.5}}
\put(4.9,2.1){$v_0$}
\put(4,1.9){$v_1$}
\put(4.1,0.1){$v_2$}
\put(5.7,0.1){$v_3$}
\put(6.1,0.9){$v_4$}

\put(9,1){\vector(1,0){1}}
\put(9,1){\vector(0,1){1}}
\put(9,1){\vector(1,1){0.7}}
\put(9,1){\vector(-1,-1){0.6}}
\put(9,1){\vector(1,-2){0.5}}
\put(8.9,2.1){$v_0$}
\put(9.8,1.9){$v_3$}
\put(8.1,0.1){$v_1$}
\put(9.7,0.1){$v_2$}
\put(10.1,0.9){$v_4$}

\put(13,1){\vector(1,0){1}}
\put(13,1){\vector(0,1){1}}
\put(13,1){\vector(4,3){1}}
\put(13,1){\vector(-1,-2){0.4}}
\put(13,1){\vector(-3,-1){0.8}}
\put(12.9,2.1){$v_0$}
\put(14.1,1.8){$v_2$}
\put(12.3,0){$v_1$}
\put(11.8,0.5){$v_3$}
\put(14.1,1){$v_4$}
\end{picture}
\caption{The case $n=3$. Type ${\rm I}_0$, ${\rm II}_1$, ${\rm I}_1$, ${\rm II}_2$ from the left.} \label{fig:n=3}
\end{figure}

\begin{figure}[hbtp]
\setlength{\unitlength}{0.8cm}
\begin{picture}(15,2.2)
\put(0,1){\vector(1,0){1}}
\put(0,1){\vector(0,1){1}}
\put(0,1){\vector(-1,2){0.6}}
\put(0,1){\vector(-1,1){1}}
\put(0,1){\vector(-2,1){1}}
\put(0,1){\vector(-4,1){1.3}}
\put(-0.1,2.1){$v_0$}
\put(-0.7,2.4){$v_1$}
\put(-1.4,2.2){$v_2$}
\put(-1.6,1.6){$v_3$}
\put(-2,1.2){$v_4$}
\put(1.1,1){$v_5$}

\put(4,1){\vector(1,0){1}}
\put(4,1){\vector(0,1){1}}
\put(4,1){\vector(-1,1){0.7}}
\put(4,1){\vector(-1,-1){0.6}}
\put(4,1){\vector(1,-2){0.5}}
\put(4,1){\vector(1,-1){0.6}}
\put(3.9,2.1){$v_0$}
\put(3,1.9){$v_1$}
\put(3.1,0.1){$v_2$}
\put(4.6,-0.3){$v_3$}
\put(5,0.1){$v_4$}
\put(5.1,0.9){$v_5$}

\put(8,1){\vector(1,0){1}}
\put(8,1){\vector(0,1){1}}
\put(8,1){\vector(-1,1){0.7}}
\put(8,1){\vector(-1,-1){0.6}}
\put(8,1){\vector(1,-2){0.5}}
\put(8,1){\vector(1,1){0.6}}
\put(7.9,2.1){$v_0$}
\put(7,1.9){$v_1$}
\put(7.1,0.1){$v_2$}
\put(8.6,-0.3){$v_3$}
\put(8.7,1.7){$v_4$}
\put(9.1,0.9){$v_5$}

\put(12,1){\vector(1,0){1}}
\put(12,1){\vector(0,1){1}}
\put(12,1){\vector(-2,-1){1}}
\put(12,1){\vector(-1,-2){0.4}}
\put(12,1){\vector(1,-2){0.5}}
\put(12,1){\vector(1,1){0.6}}
\put(11.9,2.1){$v_0$}
\put(11.2,-0.1){$v_1$}
\put(12.6,-0.3){$v_2$}
\put(12.7,1.7){$v_3$}
\put(10.4,0.4){$v_4$}
\put(13.1,0.9){$v_5$}

\put(16,1){\vector(1,0){1}}
\put(16,1){\vector(0,1){1}}
\put(16,1){\vector(1,1){0.7}}
\put(16,1){\vector(-1,-2){0.4}}
\put(16,1){\vector(4,1){1}}
\put(16,1){\vector(-2,-1){0.8}}
\put(15.9,2.1){$v_0$}
\put(16.8,1.9){$v_2$}
\put(15.2,-0.2){$v_1$}
\put(14.6,0.3){$v_3$}
\put(17.1,1.3){$v_4$}
\put(17.1,0.7){$v_5$}

\end{picture}
\caption{The case $n=4$. Type ${\rm I}_0$, ${\rm II}_1$, ${\rm I}_1$, ${\rm II}_2$, ${\rm I}_2$ from the left.} \label{fig:n=4}
\end{figure}
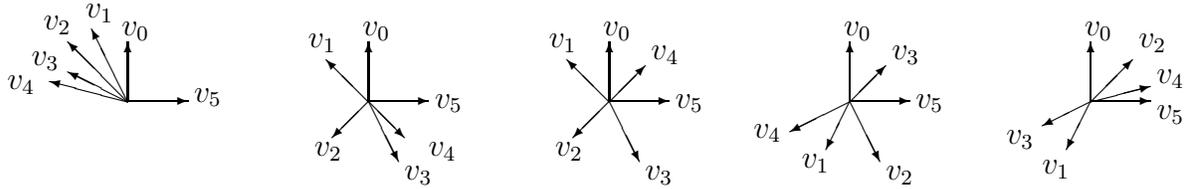

If $v$ and $v'$ are vector sequences of the same type, then it is not difficult to see that $v$ can be deformed to $v'$ continuously in the space of $V$ so that $v$ and $v'$ are in the same connected component of $V/\R^*$. In Type ${\rm I}_k$ we have $S(v)=n$ and hence $\Sign(A)=4k-n$ by Theorem~\ref{theo:main}. Similarly, in Type ${\rm II}_k$ we have $S(v)=n+2$ and hence $\Sign(A)=4k-n-2$. These show that $\Sign(A)$ distinguishes the types, proving the proposition. 
\end{proof}

\begin{remark}
We can consider a complex analogue of the configuration space \eqref{eq:V}, that is 
\[
V_\C=\left\{v=\left(\begin{pmatrix}a_1\\
b_1\end{pmatrix},\dots,\begin{pmatrix}a_n\\
b_n\end{pmatrix}\right)\in (\C^2)^n\ \middle|\ a_ib_i\in\R\ \text{for $\forall i$},\  a_1b_n\prod_{i=1}^{n-1}(\bar{a}_i\bar{b}_{i+1}-a_{i+1}b_i)\not=0\right\}
\]
and associate a Hermitian matrix $A$ like in \eqref{eq:matrix_A} to $v\in V_\C$ (to be precise, we take complex conjugate in the lower triangular part of $A$). The inverse of $A$ is a tridiagonal Hermitian matrix with nonzero superdiagonal. We denote by $W_\C$ the set of all nonsingular tridiagonal Hermitian matrices with nonzero superdiagonal.  Then the map sending $v\in V_\C$ to $A^{-1}\in W_\C$ induces a homeomorphism from $V_\C/\C^*$ to $W_\C$, where the action of $\C^*$ on $V_\C$ is given by $\begin{pmatrix}a_i\\
b_i\end{pmatrix}\to \begin{pmatrix}ca_i\\
c^{-1}b_i\end{pmatrix}$ $(c\in \C^*)$ for every $i$.  Entries in the superdiagonal of elements of $W_\C$ are in $\C^*$ and $\C^*$ is connected, so it turns out that the connected components of $V_\C/\C^*$ or $W_\C$ are distinguished by $\Sign(A)$ and they have $n+1$ many connected components.  
\end{remark}

\section{A geometric relation between $A$ and $A^{-1}$} \label{sect:4}

Let $X$ be a connected Poincar\'e duality space of dimension $4d$, i.e. $X$ is a connected topological space with a class $[X]\in H^{4d}(X;\R)$ such that $H^*(X;\R)$ is a finite dimensional vector space and the map 
\[
\cap [X]\colon H^q(X;\R)\to H_{4d-q}(X;\R) 
\]
is an isomorphism for any $q$. Then the pairing 
\[
H^{2d}(X;\R)\times H^{2d}(X;\R) \to \R \qquad (\beta,\gamma)\to \langle \beta\cup\gamma,[X]\rangle
\]
is a nondegenerate symmetric bilinear form. 

Let $n=\dim_\R H_{2d}(X;\R)$ and choose a basis $\alpha_1,\dots,\alpha_n$ of $H_{2d}(X;\R)$. Then there are two natural bases of $H^{2d}(X;\R)$. One is the Kronecker dual basis $\alpha_1^*,\dots,\alpha_n^*$, that is, $\langle \alpha_i^*,\alpha_j\rangle=\delta_{ij}$. The other is the Poincar\'e dual basis $\alpha_1^\vee,\dots,\alpha_n^\vee$, that is, $\alpha_i^\vee\cap [X]=\alpha_i$. 

\begin{lemma} \label{lemm:relation}
The inverse of the symmetric matrix with $\langle \alpha_i^*\cup\alpha_j^*,[X]\rangle$ as $(i,j)$ entry is the symmetric matrix with $\langle \alpha_i^\vee\cup\alpha_j^\vee,[X]\rangle$ as $(i,j)$ entry. 
\end{lemma}

\begin{proof}
We abbreviate $\langle \beta\cup\gamma,[X]\rangle$ as $\beta\gamma$ for $\beta,\gamma\in H^{2d}(X,\R)$. We denote the symmetric matrix $(\alpha_i^*\alpha_j^*)$ (resp. $(\alpha_i^\vee\alpha_j^\vee)$) by $A$ (resp. $B$). What we have to prove is $B=A^{-1}$. 

Since $\alpha_1^*,\dots,\alpha_n^*$ and $\alpha_1^\vee,\dots,\alpha_n^\vee$ are both bases of $H^{2d}(X;\R)$, one has a nonsingular matrix $P$ of order $n$ such that
\begin{equation} \label{eq:matrix_P}
\begin{pmatrix} \alpha_1^*\\
\vdots\\
\alpha_n^*\end{pmatrix}=P\begin{pmatrix} \alpha_1^\vee\\
\vdots\\
\alpha_n^\vee\end{pmatrix}
\end{equation}
Therefore
\begin{equation*} \label{eq:dual}
\begin{pmatrix} \alpha_1^*\\
\vdots\\
\alpha_n^*\end{pmatrix}\begin{pmatrix}\alpha_1,\dots,\alpha_n\end{pmatrix}=P\begin{pmatrix} \alpha_1^\vee\\
\vdots\\
\alpha_n^\vee\end{pmatrix}\begin{pmatrix}\alpha_1,\dots,\alpha_n\end{pmatrix}
=P\begin{pmatrix} \langle\alpha_1^\vee,\alpha_1\rangle&\dots&\langle\alpha_1^\vee,\alpha_n\rangle\\
\vdots&\vdots&\vdots\\
\langle\alpha_n^\vee,\alpha_1\rangle&\dots&\langle\alpha_n^\vee,\alpha_n\rangle\end{pmatrix}.
\end{equation*}
Here the left hand side is the identity matrix while the last matrix above is $B$ because 
\[
\langle \alpha_i^\vee,\alpha_j\rangle=\langle \alpha_i^\vee,\alpha_j^\vee\cap [X]\rangle =\langle\alpha_i^\vee\cup\alpha_j^\vee,[X]\rangle. 
\]
Therefore, $B=P^{-1}$ and since $B^T=B$ we have $P^T=P$. 

Similarly, it follows from \eqref{eq:matrix_P} that 
\[
A=\begin{pmatrix} \alpha_1^*\\
\vdots\\
\alpha_n^*\end{pmatrix}\begin{pmatrix}\alpha_1^*,\dots,\alpha_n^*\end{pmatrix}=P\begin{pmatrix} \alpha_1^\vee\\
\vdots\\
\alpha_n^\vee\end{pmatrix}\begin{pmatrix}\alpha_1^\vee,\dots,\alpha_n^\vee\end{pmatrix}P^T=PBP^T=P
\]
where the last identity is because $B=P^{-1}$ and $P^T=P$. Thus $P=A$ and hence $B=A^{-1}$ because $B=P^{-1}$ as observed above. 
\end{proof}

\section{A geometric realization of the matrix $A$} \label{sect:5}

When $v_i$'s in \eqref{eq:vi} are (primitive) integer vectors and the sequence $v=(v_0,v_1,\dots,v_n,v_{n+1})$ rotates around the origin once counterclockwise, the $v$ associates a two dimensional complete simplicial fan by taking cones spanned by consecutive vectors and hence the $v$ associates a compact toric orbifold $X$ of complex dimension two. Let $\alpha_i$ be an element of $H_2(X;\R)$ defined by the divisor $X_i$ corresponding to $v_i$ for $1\le i\le n$. Then $\alpha_1,\dots,\alpha_n$ form a basis of $H_2(X;\R)$ and it is proved in \cite{fu-so-so23, fu-so-so24} that the matrix $A$ in \eqref{eq:matrix_A} is $(\langle \alpha_i^*\cup\alpha_j^*,[X]\rangle)$ while $A^{-1}$ is the matrix $(\langle \alpha_i^\vee\cup\alpha_j^\vee,[X]\rangle)$ which is the intersection matrix of the divisors $X_1,\dots,X_n$. As is easily observed, $X_i$ and $X_j$ do not intersect when $|i-j|\ge 2$. Therefore, $\langle\alpha_i^\vee\cup\alpha_j^\vee,[X]\rangle=0$ when $|i-j|\ge 2$ and hence $A^{-1}$ is tridiagonal. 

Suppose that $v_1,\dots,v_n$ are (not necessarily primitive) integer vectors with $\det A\not=0$, where the sequence $v=(v_0, v_1,\dots,v_n,v_{n+1})$ may rotate around the origin many times and may go back and forth. We associate an omnioriented quasitoric orbifold $X$ to the $v$ as follows. 
Let $Q$ be an $(n+2)$-gon given by 
\[
Q=\{ \x\in \R^2\mid (\n_i,\x)+d_i\ge 0\quad (i=0,1\dots,n+1)\} 
\]
with nonzero vectors $\n_i\in\R^2$ and $d_i\in\R$, where $(\ ,\ )$ denotes the standard scalar product on $\R^2$. Then the moment-angle manifold $\mathcal{Z}_Q$ is defined as 
\[
\mathcal{Z}_Q=\{(\x,\z)\in \R^2\times \C^{n+2}\mid (\n_i,\x)+d_i=|z_i|^2\quad (i=0,1,\dots,n+1)\} 
\]
where $\z=(z_0,z_1,\dots,z_{n+1})$. It is an orientable connected closed smooth manifold of dimension $n+4$ (\cite[Chapter 6]{bu-pa15}). The $\mathcal{Z}_Q$ has a natural action of $(S^1)^{n+2}\subset (\C^*)^{n+2}$, that is coordinatewise multiplication on $\C^{n+2}$. The sequence $v$ associates a surjective homomorphism $\mathcal{V}\colon (S^1)^{n+2}\to (S^1)^2=T$ defined by 
\[
\begin{split}
\mathcal{V}(h_0,h_1,\dots,h_{n+1})&=(h_0^{a_0}h_1^{a_1}\cdots h_n^{a_n}h_{n+1}^{a_{n+1}},\ h_0^{b_0}h_1^{b_1}\cdots h_n^{b_n}h_{n+1}^{b_{n+1}})\\
&=(h_1^{a_1}\cdots h_n^{a_n}h_{n+1},\ h_0h_1^{b_1}\cdots h_n^{b_n})
\end{split}
\]
where $(a_0,b_0)=(0,1)$ and $(a_{n+1},b_{n+1})=(1,0)$ as before. The kernel of $\mathcal{V}$ is an $n$-dimensional subtorus of $(S^1)^{n+2}$ and its action on $\mathcal{Z}_Q$ is almost free. Our quasitoric orbifold $X$ is $\mathcal{Z}_Q/\ker\mathcal{V}$ with the induced action of $(S^1)^{n+2}/\ker\mathcal{V}=T$, where the last identification is through $\mathcal{V}$. 

Let $S_i$ denote the $i$-th factor of $(S^1)^{n+2}$ for $i=0,1,\dots,n+1$. 
As usual, the suffix $n+2$ is understood to be $0$ in the following. For $i=0,1,\dots,n+1$, we set 
\[
\widetilde{U}_i=\{(\x,\z)\in \mathcal{Z}_Q\mid z_j\not=0\ \text{for $j\not=i,i+1$}\},\quad
V_i=\tilde{\varphi}_i(\widetilde{U}_i)
\]
where $\tilde{\varphi}_i\colon \mathcal{Z}_Q\to \C_i^2=\C^2$ is defined by $\tilde{\varphi}_i(\x,\z)=(z_i,z_{i+1})$. We note that $S_i\times S_{i+1}$ is acting on $\C^2_i$ by coordinatewise multiplication and $V_i$ is an invariant open subset of $\C^2_i$. We define a subgroup $G_i$ of $S_i\times S_{i+1}$ by 
\[
G_i=\{(h_i,h_{i+1})\in S_i\times S_{i+1}\mid h_i^{a_i}h_{i+1}^{a_{i+1}}=h_i^{b_i}h_{i+1}^{b_{i+1}}=1\}.
\]
The order of $G_i$ is $|a_ib_{i+1}-a_{i+1}b_i|=|\det(v_i,v_{i+1})|$. One can check that the map $\tilde{\varphi}_i$ induces a homeomorphism 
\begin{equation*} \label{eq:Ui}
\varphi_i\colon U_i:=\widetilde{U}_i/\ker\mathcal{V}\to V_i/G_i.
\end{equation*}
The set $\{U_i\}_{i=0}^{n+1}$ is an open covering of $X$, so $\{(U_i,V_i,G_i,\varphi_i)\}_{i=0}^{n+1}$ defines an orbifold structure on $X$. We think of $X$ as an orbifold with this orbifold structure. 

\begin{remark}
When $\{v_i,v_{i+1}\}$ forms a basis of $\Z^2$ for any $0\le i\le n$, the quasitoric orbifold $X$ is smooth. In this case $X$ is constructed in \cite[Section 5]{masu99} by a different method and if the sequence $v$ rotates around the origin counterclockwise possibly many times, the $X$ admits an almost complex structure. 
\end{remark} 

The characteristic suborbifold $X_i$ corresponding to $v_i$ is $(\mathcal{Z}_Q\cap \{z_i=0\})/\ker\mathcal{V}$. It is fixed pointwise by the $S^1$-subgroup $\mathcal{V}(S_i)$ of $T$. We orient the normal bundle of $X_i$ though the rotation by the action of $\mathcal{V}(S_i)$ on it, where the orientation on $S_i\ (\subset \C)$ is the standard one. Since $\{v_{n+1},v_0\}$ forms the standard basis of $\Z^2$, the group $G_{n+1}$ above is trivial and the neighborhood $U_{n+1}$ of the $T$-fixed point $X_{n+1}\cap X_0$ is homeomorphic to $V_{n+1}\subset \C^2_{n+1}=\C^2$ though $\varphi_{n+1}$.  We give an orientation on $X$ induced from $\C^2$ though $\varphi_{n+1}$.  Thus, this orientation on $X$ and the orientation on the normal bundle of $X_i$ determines a compatible orientation on each $X_i$, so $X$ is omnioriented. 

Let $\alpha_i$ $(0\le i\le n+1)$ be the element of $H_2(X;\R)$ determined by the characteristic suborbifold $X_i$.  We denote the Poincar\'e dual of $\alpha_i$ by $\alpha_i^\vee$ as before. 

\begin{theorem} \label{theo:main2}
$A^{-1}=(\langle \alpha_i^\vee\cup\alpha_j^\vee,[X]\rangle)_{1\le i,j\le n}$. 
\end{theorem}

Theorem~\ref{theo:main2_intro} in the introduction follows from Theorem~\ref{theo:main2} and Proposition~\ref{prop:inverse_A}.

The rest of this section is devoted to the proof of Theorem~\ref{theo:main2}. 
Let $\xi_i\in H^2_T(X;\Q)$ $(0\le i\le n+1)$ be the image of $1\in H^0_T(X_i;\Q)$ by the equivariant Gysin map $H^0_T(X_i;\Q)\to H^2_T(X;\Q)$. The restriction of $\xi_i$ to $H^2(X;\Q)$ is $\alpha_i^\vee$. 
We identify $\Z^2$ with $H_2(BT;\Z)$ and think of $v_i$ as an element of $H_2(BT;\Z)$. 

\begin{lemma}[{\cite[Lemmas 12,4 and 12.5]{ha-ma03}}] \label{lemm:dual}
$\{\xi_i|_{p_i}, \xi_{i+1}|_{p_i}\}$ $(p_i=X_i\cap X_{i+1})$ is dual to $\{v_i, v_{i+1}\}$ for $i=0,\dots,n+1$ where $\xi_{n+2}=\xi_0$ and $X_{n+2}=X_0$. 
\end{lemma}

\begin{lemma}[{\cite[Lemma 12.6]{ha-ma03}}] \label{lemm:fundamental}
We have 
\begin{equation} \label{eq:fundamental}
\pi^*(u)=\sum_{i=0}^{n+1}\langle u,v_i\rangle\xi_i\quad\text{for $u\in H^2(BT;\Q)$}
\end{equation}
where $\pi^*\colon H^*(BT;\Q)\to H^*_T(X;\Q)$ is the homomorphism induced from $\pi\colon X\to pt$. 
\end{lemma}

\begin{proof}
This is an easy consequence of Lemma~\ref{lemm:dual}. Indeed, if we take $u$ to be $\xi_i|_{p_i}$ or $\xi_{i+1}|_{p_i}$, then \eqref{eq:fundamental} restricted to $p_i$ holds by Lemma~\ref{lemm:dual}. Since $\{\xi_i|_{p_i},\xi_{i+1}|_{p_i}\}$ is a basis of $H^*(BT;\Q)$, \eqref{eq:fundamental} restricted to $p_i$ holds for any $u\in H^2(BT;\Q)$. This implies \eqref{eq:fundamental} because the restriction map $H^*_T(X;\Q)\to H^*_T(X^T;\Q)$ is injective and $X^T=\{p_i\}_{i=0}^{n+1}$. 
\end{proof}

Restricting \eqref{eq:fundamental} to $H^*(X;\Q)$, we obtain 
\begin{equation*} \label{eq:linear_relation}
\sum_{i=0}^{n+1}\langle u,v_i\rangle \alpha_i^\vee=0\quad\text{for $u\in H^2(BT;\Q)$}.
\end{equation*}
For each $1\le i\le n$, we take the cup product with $\alpha_i^\vee$ on the both sides of this identity and evaluate on the fundamental class $[X]$. Then, since the $X_i$ does not intersect with $X_j$ when $|i-j|\ge 2$, we obtain 
\begin{equation} \label{eq:linear_relation_2}
\langle u,v_{i-1}\rangle \alpha_{i-1}^\vee\alpha_i^\vee+ \langle u,v_{i}\rangle(\alpha_{i}^\vee)^2+\langle u,v_{i+1}\rangle\alpha_{i+1}^\vee\alpha_i^\vee=0
\end{equation} 
where $\beta\gamma$ for $\beta,\gamma\in H^2(X;\Q)$ denotes $\langle \beta\cup\gamma,[X]\rangle$ as before. 
Since the identity \eqref{eq:linear_relation_2} holds for any $u$, it reduces to 
\begin{equation} \label{eq:linear_relation_among_three}
(\alpha_{i-1}^\vee\alpha_i^\vee) v_{i-1}+(\alpha_{i}^\vee)^2 v_i+(\alpha_{i+1}^\vee\alpha_i^\vee) v_{i+1}=0.
\end{equation}
Therefore, the $s_i$ in \eqref{eq:inverse_A} is given by 
\begin{equation} \label{eq:di}
s_i={(\alpha_i^\vee)^2}/\det(v_i,v_{i+1})\alpha_i^\vee\alpha_{i+1}^\vee
\end{equation}
because $s_i=-\det(v_{i-1},v_{i+1})/\det(v_{i-1},v_i)\det(v_i,v_{i+1})$ and 
\[
\begin{split}
\det(v_{i-1},v_{i+1})&=\det\left(v_{i-1},-\frac{(\alpha_{i-1}^\vee\alpha_i^\vee)v_{i-1}+(\alpha_i^\vee)^2v_i}{\alpha_i^\vee\alpha_{i+1}^\vee}\right)\\
&=-(\alpha_i^\vee)^2\det(v_{i-1},v_i)/\alpha_i^\vee\alpha_{i+1}^\vee
\end{split}
\]
where we used \eqref{eq:linear_relation_among_three} at the former identity above. 

\begin{lemma} \label{lemm:intersection_number}
$\det(v_i,v_{i+1})\alpha_i^\vee\alpha_{i+1}^\vee=1$. 
\end{lemma}

\begin{proof}
Since $\{\xi_i|_{p_i},\xi_{i+1}|_{p_i}\}$ is dual to $\{v_i,v_{i+1}\}$ by Lemma~\ref{lemm:dual}, it follows from \cite[Theorem 2.1]{mein98}\footnote{The integration formula is \emph{defined} in the category of multi-fans in \cite[the formula above Lemma 8.4]{ha-ma03}.} that the integration formula for the equivariant Gysin map $\pi_!\colon H^4_T(X;\Q)\to H^{0}_T(pt;\Q)=\Q$ is given by 
\begin{equation}\label{eq:integration_formula}
\pi_!(\beta)=\sum_{i=0}^{n+1}\frac{\beta|_{p_i}}{\det(v_i,v_{i+1})\xi_i|_{p_i}\xi_{i+1}|_{p_i}} \quad\text{for $\beta\in H^4_T(X;\Q)$}. 
\end{equation}
Since $(\xi_i\cup\xi_{i+1})|_{p_j}=0$ unless $j=i$, we have $$\pi_!(\det(v_i,v_{i+1})\xi_i\cup\xi_{i+1})=1.$$ Restricting this identity to ordinary cohomology, we obtain the lemma because $\pi_!$ and $\xi_i$ in ordinary cohomology are respectively the evaluation on $[X]$ and $\alpha_i^\vee$ for any $0\le i\le n+1$. 
\end{proof}

By Lemma~\ref{lemm:intersection_number} and \eqref{eq:di}, $m_i^{-1}$ and $s_i$ in \eqref{eq:inverse_A} are respectively $\alpha_i^\vee\alpha_{i+1}^\vee$ and $(\alpha_i^\vee)^2$, so the matrix in \eqref{eq:inverse_A} is the intersection matrix $(\alpha_i^\vee\alpha_j^\vee)$. This proves Theorem~\ref{theo:main2}.

\end{document}